\documentstyle[11pt,amsmath,amsthm,amsfonts,color,verbatim]{article}

\newtheorem{theorem}{Theorem}[section]

\newtheorem{ex}[theorem]{Example}

\newtheorem{pro}[theorem]{Proposition}

{

\begin{document}

\title{Averaged alternating reflections in geodesic spaces}
\author{Aurora Fern\'andez-Le\'on$^{1}$, Adriana Nicolae$^{2,3}$}
\date{}
\maketitle

\begin{center}
{\footnotesize
$^{1}$Dpto. de An\'alisis Matem\'atico, Universidad de Sevilla, P.O. Box 1160, 41080-Sevilla, Spain
\ \\
$^{2}$ Simion Stoilow Institute of Mathematics of the Romanian Academy, P.O. Box 1-764, 014700 Bucharest, Romania
\ \\
$^{3}$ Department of Mathematics, Babe\c s-Bolyai University, Kog\u alniceanu 1, 400084, Cluj-Napoca, Romania
\ \\
E-mail: aurorafl@us.es, adriana.nicolae@imar.ro
}
\end{center}

\begin{abstract}
We study the nonexpansivity of reflection mappings in geodesic spaces and apply our findings to the averaged alternating reflection algorithm employed in solving the convex feasibility problem for two sets in a nonlinear context. We show that weak convergence results from Hilbert spaces find natural counterparts in spaces of constant curvature. Moreover, in this particular setting, one obtains strong convergence.
\end{abstract}

\section{Introduction}
The convex feasibility problem for two sets consists of finding a point in the intersection of two nonempty closed and convex sets provided such a point exists. This problem finds remarkable applications in applied mathematics and various branches of engineering (see, for example, \cite{BauBor96, Com96, PesCom96, Byr03}) which have motivated many researchers to focus on methods of solving this problem. In Hilbert spaces there exists a wide range of algorithms that use metric projections on the sets in order to obtain sequences of points that converge weakly or in norm (under more restrictive conditions) to a solution of this problem. One of the most famous algorithms is the alternating projection method which was developed by von Neumann \cite{vNeu33} and was recently adapted to the setting of CAT$(0)$ spaces by Ba\v c\' ak, Searston and Sims in \cite{BacSeaSim12}.

Another class of algorithms considered in this respect is based on reflections instead of projections. Given a nonempty closed and convex subset $A$ of a Hilbert space $H$, the reflection of a point $x \in H$ with respect to $A$ is the image of $x$ by the reflection mapping $R_A=2P_A-I$, where $P_A$ stands for the metric projection onto $A$ and $I$ is the identity mapping. In this work we focus on the averaged alternating reflection (AAR) method employed in solving the convex feasibility problem for two sets. Suppose $A$ and $B$ are two nonempty closed and convex subsets of a Hilbert space with nonempty intersection. The AAR method generates the following sequence for a starting point $x_0 \in H$: $x_n = T^n x_0$, where $\displaystyle T = \frac{I+R_A R_B}{2}$. This algorithmic scheme was studied by Bauschke, Combettes and Luke in \cite{bauscom1, bauscom2} not only in connection with the convex feasibility problem, but also for finding a best approximation pair of the sets $A$ and $B$ in case their intersection is empty and such a pair exists. The AAR method was later modified in \cite{bauscom3} in order to solve the problem of finding the projection of a point onto the intersection of two closed and convex sets. In fact, for the convex feasibility problem, this algorithm is a special case of one described by Lions and Mercier in \cite{LioMer79}. One obtains weak convergence of the sequence $(x_n)$ to a fixed point of the mapping $T$ and the projection of this point onto the set $B$ lies in the intersection of $A$ and $B$.

Here we are interested in studying the AAR method in geodesic spaces. However, several difficulties arise when considering this algorithm in a nonlinear setting. First of all one needs to find an appropriate definition for the reflection mapping. In order to guarantee the existence of the reflection of a point in the space, we consider spaces with the geodesic extension property. Moreover, the reflection mapping is not always unique. A second difficulty consists in guaranteeing certain properties of this mapping. In Hilbert spaces, the proof of the convergence of the AAR method relies on the nonexpansivity of the reflection mapping which yields the firm nonexpansivity of the mapping $T$.

In this paper we prove that the reflection mapping is nonexpansive in spaces of constant curvature and justify why it fails to be nonexpansive in the broad setting of CAT$(0)$ spaces. We also analyze the behavior of reflection mappings in slightly more general settings, namely gluings of model spaces. Likewise, we study the convergence of the AAR method in spaces of constant curvature proving strong convergence in this case. Furthermore, we include a rate of asymptotic regularity for the AAR method.

This work is partly motivated by a communication of Ian Searston given during the 10th International Conference on Fixed Point Theory and its Applications where the problem of studying the nonexpansivity of the reflection mapping in geodesic spaces was raised.

\section{Preliminaries}
A metric space $(X,d)$ is said to be a {\it (uniquely) geodesic space} if every two points $x$ and
$y$ of $X$ are joined by a (unique) geodesic, i.e., a map $c \colon [0,l]\subseteq {\mathbb R}\to X$ such that $c(0)=x$,
$c(l)=y$, and $d(c(t),c(t^{\prime}))=|t-t^{\prime}|$ for all $t,t^{\prime} \in [0,l]$. The image $c([0,l])$ of a geodesic forms a {\it geodesic segment} which joins $x$ and $y$ and is not necessarily unique. If no confusion arises, we use $[x,y]$ to denote a geodesic segment
joining $x$ and $y$. A point $z$ in $X$ belongs
to a geodesic segment $[x,y]$ if and only if there exists $t\in [0,1]$ such that $d(x,z)=td(x,y)$ and $d(y,z)=(1-t)d(x,y)$ and we write $z=(1-t)x+ty$ for simplicity. Notice that this point may not be unique. A subset of $X$ is said to be convex if it contains any geodesic segment that joins every two points of it. A {\it geodesic triangle} $\triangle(x,y,z)$ consists of three points $x,y,z \in X$ (the vertices of $\triangle$) and three geodesic segments joining each pair of vertices (the {\it edges} of $\triangle$). A {\it geodesic line} in $X$ is a subset of $X$ isometric to ${\mathbb R}$. A geodesic space has the {\it geodesic extension property} if each geodesic segment is contained in a geodesic line. More on geodesic metric spaces can be found for instance in \cite{brha,papa}.

The {\it metric} $d \colon X \times X \rightarrow {\mathbb R}$ is said to be {\it convex} if for any $x, y, z \in X$ one has
$$d(x, (1 - t)y + tz) \leq (1 - t)d(x, y) + td(x, z) \text{ for all } t \in [0, 1].$$

A geodesic space $(X,d)$ is {\it Busemann convex} (introduced in \cite{bu}) if given any pair of geodesics $c_1 : [0, l_1] \to X$ and $c_2 : [0,l_2] \to X$ one has
$$
d(c_1(tl_1),c_2(tl_2)) \le (1-t)d(c_1(0),c_2(0)) + td(c_1(l_1),c_2(l_2)) \mbox{ for all } t \in [0,1].
$$
It is well-known that Busemann convex spaces are uniquely geodesic and with convex metric.

A very important class of geodesic metric spaces are CAT($k$) spaces (where $k\in \mathbb{R}$), that is, metric spaces of
curvature uniformly bounded above by $k$ in the sense of Gromov. CAT($k$) spaces are defined in terms of comparisons with the model spaces $M_k^n$, which are the complete, simply connected, Riemannian $n$-manifolds of constant sectional curvature $k$. Since these model spaces are of essential importance in this work we give their definition directly as metric spaces and recall some of their properties. For a thorough treatment of such spaces and related topics the reader can check \cite{brha,gr}.

The {\it $n$-dimensional sphere ${\mathbb{S}}^n$} is the set $\left\{x \in {\mathbb{R}}^{n + 1} : (x \mid x) = 1\right\}$, where $(\cdot\mid \cdot)$ is the Euclidean scalar product. Define $d: {\mathbb{S}}^n \times {\mathbb{S}}^n \to \mathbb{R}$ by assigning to each $(x,y) \in {\mathbb{S}}^n \times {\mathbb{S}}^n$ the unique number $d(x,y) \in [0,\pi]$ such that $\cos d(x,y) = (x\mid y)$. Then $({\mathbb{S}}^n,d)$ is a metric space called the spherical space. This is a geodesic space and if $d(x,y) < \pi$ then there is a unique geodesic joining $x$ and $y$. Also, balls of radius smaller than $\pi/2$ are convex.  The {\it spherical law of cosines} states that in a spherical triangle with vertices $x,y,z \in {\mathbb{S}}^n$ and $\gamma$ the spherical angle between the geodesic segments $[x,y]$ and $[x,z]$ we have
\[\cos d(y,z) = \cos d(x,y) \cos d(x,z) + \sin d(x,y) \sin d(x,z) \cos \gamma.\]

For $u,v \in {\mathbb{R}}^{n+1}$, consider the quadratic form given by $\langle u \mid v\rangle  = -u_{n+1}v_{n+1} +\sum_{i=1}^nu_iv_i$. The {\it hyperbolic $n$-space} ${\mathbb{H}}^n$ is the set $\{u=(u_1,u_2,...,u_{n+1})\in {\mathbb{R}}^{n+1}:\; \langle u \mid u\rangle=-1, u_{n+1}>0 \}$. Then ${\mathbb{H}}^n$ is a metric space with the hyperbolic distance $d:{\mathbb{H}}^n \times {\mathbb{H}}^n \to \mathbb{R}$ assigning to each $(x,y) \in {\mathbb{H}}^n \times {\mathbb{H}}^n$ the unique number $d(x,y) \ge 0$ such that $\cosh d(x,y)=-\langle x \mid y\rangle$. The hyperbolic space is uniquely geodesic and all its balls are convex. The {\it hyperbolic law of cosines} states that in a hyperbolic triangle with vertices $x,y,z \in {\mathbb{H}}^n$ and $\gamma$ the hyperbolic angle between the geodesic segments $[x,y]$ and $[x,z]$ we have
\[\cosh d(y,z) = \cosh d(x,y) \cosh d(x,z) - \sinh d(x,y) \sinh d(x,z) \cos \gamma.\]

Let $k\in \mathbb{R}$ and $n \in \mathbb{N}$. The classical {\it model spaces} $M_k^n$ are defined as follows: if $k > 0$, $M_k^n$
is obtained from the spherical space ${\mathbb S}^n$ by multiplying the spherical distance with $1/\sqrt k$; if
$k= 0$, $M_0^n$ is the $n$-dimensional Euclidean space ${\mathbb E}^n$; and if $k < 0$, $M_k^n$ is obtained from
the hyperbolic space ${\mathbb H}^n$  by multiplying the hyperbolic distance with $1/\sqrt{-k}$. The model spaces inherit their geometrical properties from the three Riemaniann manifolds that define them. Thus, if $k < 0$, $M_k^n$ is uniquely geodesic, balls are convex and we have a counterpart of the hyperbolic law of cosines. If $k>0$,
there is a unique geodesic segment joining $x,y\in M_k^n$ if and only if $d(x,y)<\pi/\sqrt{k}$. Moreover, closed balls of
radius smaller than $\pi/(2\sqrt{k})$ are convex and we have a counterpart of the spherical law of cosines. We denote the {\it diameter of $M_k ^n$} by $D_k$. More precisely, for $k > 0$, $D_k = \pi/\sqrt{k}$ and for $k \le 0$, $D_k = \infty$.

Now we briefly introduce CAT($k$) spaces. Let $(X,d)$ be a geodesic space. A {\it $k$-comparison triangle}
for a geodesic triangle $\triangle (x_1,x_2,x_3)$ in $(X,d)$ is a triangle $\triangle(\bar{x}_1,\bar{x}_2,\bar{x}_3)$ in $M_k^2$ such
that $d_{M_k^2}(\bar{x}_i,\bar{x}_j)=d(x_i,x_j)$ for $i,j\in\{1,2,3\}$. For $k$ fixed, $k$-comparison triangles of geodesic triangles (having perimeter less than $2D_k$) always exist and are unique up to isometry.

A geodesic triangle $\triangle$ in $X$ is said to satisfy the {\it CAT$(k)$ inequality} if, given $\bar{\triangle}$ a $k$-comparison
triangle for $\triangle$, for all $x,y\in \triangle$
$$
d(x,y)\le d_{M_k^2}(\bar{x},\bar{y}),
$$
where $\bar{x},\bar{y}\in\bar{\triangle}$ are the comparison points of $x$ and $y$, respectively.

If $k \le 0$, a {\it CAT($k$) space} is a geodesic space for which every geodesic triangle satisfies the CAT($k$) inequality. If $k > 0$, a metric space is called a CAT($k$) space if every two points at distance less than $D_k$ can be joined by a
geodesic and every geodesic triangle having perimeter less than $2D_k$ satisfies the CAT($k$) inequality.

Let $X$ be a metric space and $C \subseteq X$. We define the {\it distance of a point} $z \in X$ to $C$ by $\displaystyle \mbox{dist}(z,C) = \inf_{y \in C}d(z,y)$. The {\it metric projection} $P_C$ onto $C$ is the mapping
\[P_C(z)=\{ y \in C : d(z,y)=\mbox{dist}(z,C)\} \mbox{ for every } z\in X.\]

The following proposition gathers important properties of the metric projection in the setting of CAT($0$) spaces. In \cite{esfe}, a counterpart of this result was given in the setting of CAT($k$) spaces with $k>0$.

\begin{pro}[\cite{brha}, Proposition 2.4, p. 176] \label{projection}
Let $X$ be a complete CAT$(0)$ space, $x\in X$ and $C\subset X$
nonempty closed and convex. Then
the following hold:
\begin{enumerate}
\item The metric projection $P_C(x)$ of $x$ onto $C$ is a
singleton.
\item If $y\in [x,P_C(x)]$, then $P_C(x)=P_C(y)$.
\item If $x\notin C$ and $y\in C$ with $y\neq P_C(x)$ then
$\angle_{P_C(x)}(x,y)\geq \pi/2$.
\item The mapping $P_C$ is a nonexpansive retraction from $X$ onto $C$. Further, the mapping $H : X \times [0,1] \rightarrow X$ associating to $(x,t)$ the point at distance $td(x,P_C(x))$ on the geodesic $[x,P_C(x)]$ is a continuous homotopy from the identity map of $X$ to $P_C$.
\end{enumerate}
\end{pro}

We give next a concept of convergence in metric spaces. Let $(X,d)$ be a metric space, $(x_n) \subseteq X$ a bounded sequence and $x\in X$. The \emph{asymptotic radius} of $(x_n)$ is given by
$$r((x_n))=\inf \left\{\limsup_{n \to \infty}d(x,x_n): x\in X\right\}$$
and the \emph{asymptotic center} of $(x_n)$ is the set
$$A((x_n))=\left\{x \in X : \limsup_{n \to \infty}d(x,x_n) = r((x_n))\right\}.$$
The sequence $(x_n)$ in $X$ is said to {\it $\Delta$-converge} to $x\in X$ if $x$ is the unique asymptotic center of $(u_n)$ for every subsequence $(u_n)$ of $(x_n)$. This notion of convergence coincides with the weak convergence in Hilbert spaces. For more on this topic and other concepts of weak convergence in metric spaces the reader may see \cite{Jos94, so, kipa, esfe, bijan}.

\section{Main results}
\subsection{Reflecting in geodesic spaces}
Let $(X,d)$ be a uniquely geodesic space with the geodesic extension property. Let $C$ be a nonempty closed and convex subset of $X$ and suppose the metric projection onto $C$ is well-defined and singlevalued. The reflection of a point $x \in X$ with respect to $C$ can be any point $z$ in a geodesic line containing the geodesic segment $[x,P_Cx]$ for which $\displaystyle P_Cx= \frac{x+z}{2}$.

If $X$ has no bifurcating geodesics, then geodesics can be extended in a unique way so the reflection of a point is uniquely determined. Thus, in such a setting, the reflection mapping $R_C$ which assigns to each point its reflection is well-defined and singlevalued. Recall that two geodesics bifurcate if they have a common endpoint and coincide on an interval, but one is not an extension of the other. Note however that one may also define singlevalued reflection mappings in spaces failing this property.

We study next the reflection mapping in a nonlinear setting. By \cite[Example 22.1]{gore}, we know that in the complex Hilbert ball the reflection mapping fails to be nonexpansive which yields that in CAT$(0)$ spaces this mapping is not necessarily nonexpansive in general. However, in the particular setting of model spaces, the reflection mapping proves to be nonexpansive. This shows that in model spaces, as in the case of Hilbert spaces, the metric projection is an averaged mapping with constant $1/2$.

\begin{pro}\label{model}
Let $k \in \mathbb{R}$ and $n \in \mathbb{N}$. Suppose $C$ is a nonempty closed and convex subset of $M_k^n$ and $x,y\in M_k^n$ such that $\emph{dist}(x,C)$, $\emph{dist}(y,C)< D_k /2$. Then,
$$d(R_Cx,R_Cy)\leq d(x,y).$$
\end{pro}
\begin{proof}
First we consider the case $k=-1$. For simplicity, denote $c_x=P_C x$, $c_y=P_Cy$, $x^\prime=R_Cx$ and $y^\prime=R_Cy$. Let $\gamma = \angle_{c_x}(x,c_y)$ and $\gamma^\prime=\angle_{c_y}(y,c_x)$. Notice that, by Proposition \ref{projection}, $\gamma, \gamma^\prime \geq \pi/2$. Consider the geodesic triangles $\triangle(x,c_x,c_y)$ and $\triangle(x^\prime,c_x,c_y)$. By the hyperbolic cosine law we have that
$$
\cosh d(x,c_y)=\cosh d(x,c_x)  \cosh d(c_x,c_y)  - \sinh d(x,c_x)  \sinh d(c_x,c_y) \cos \gamma
$$
and
$$\cosh d(x^\prime,c_y)=\cosh d(x^\prime,c_x) \cosh d(c_x,c_y) - \sinh d(x^\prime,c_x) \sinh  d(c_x,c_y) \cos (\pi-\gamma).$$
Since $d(x,c_x)=d(x^\prime, c_x)$ and $\cos (\pi-\gamma)\geq 0$, we get $\cosh d(x,c_y) \geq \cosh d(x^\prime,c_y)$ and thus
$$d(x,c_y)\geq d(x^\prime,c_y).$$
Similarly we get $d(y,c_x) \geq d(y^\prime,c_x)$.\\
Consider now the geodesic triangles $\triangle(x,c_x,y)$ and $\triangle(x^\prime, c_x, y)$ and denote $\beta = \angle_{c_x}(x,y)$. Applying again the hyperbolic cosine law we obtain that
$$
\cosh d(x,y)=\cosh d(x,c_x)  \cosh d(c_x,y)  - \sinh  d(x,c_x)  \sinh d(c_x,y) \cos \beta
$$
and
$$\cosh d(x^\prime,y)=\cosh d(x^\prime,c_x) \cosh d(c_x,y) - \sinh d(x^\prime,c_x) \sinh  d(c_x,y) \cos (\pi-\beta).$$
Adding these two equalities we get that
\begin{equation}\label{1}
\cosh d(x,y)+\cosh d(x^\prime,y)=2 \cosh d(c_x, x) \cosh d(c_x,y).
\end{equation}
In a similar way we have that
\begin{equation}\label{2}
\cosh d(x^\prime,y^\prime)+\cosh d(x,y^\prime)=2 \cosh d(c_x, x) \cosh d(c_x,y^\prime),
\end{equation}
\begin{equation}\label{3}
\cosh d(x,y)+\cosh d(x,y^\prime)=2 \cosh d(c_y, x) \cosh d(c_y,y),
\end{equation}
\begin{equation}\label{4}
\cosh d(x^\prime,y^\prime)+\cosh d(x^\prime,y)=2 \cosh d(c_y, y) \cosh d(c_y,x^\prime).
\end{equation}
Suppose now that $d(x^\prime, y^\prime)> d(x,y)$. From (\ref{1}) and (\ref{2}),
\begin{align*}
\cosh d(x^\prime,y^\prime)+\cosh d(x,y^\prime)& = 2 \cosh d(c_x, x) \cosh d(c_x,y^\prime) \\
& \leq 2 \cosh d(c_x, x) \cosh d(c_x,y)\\
& < \cosh d(x^\prime,y^\prime)+\cosh d(x^\prime,y),
\end{align*}
which implies that $d(x,y^\prime) < d(x^\prime,y)$.
By using (\ref{3}) and (\ref{4}) in a similar way, we get that $d(x,y^\prime) > d(x^\prime,y)$, which is a contradiction and thus the result follows. The proof in any other model space $M_k^n$ follows similar patterns by using the corresponding law of cosines of each space.
\end{proof}

The above result also holds in the equivalent infinite dimensional spherical and hyperbolic spaces defined using points in $\ell_2$ (see, for instance, ${\mathbb{H}}^{\infty}$ in \cite{EspPia12} and \cite[Example 3.4]{esfe}).

In the sequel, we analyze the reflection mapping in some particular CAT($k$) spaces. More precisely, we consider the gluing of two model spaces $M_k^n$ and $M_{k^\prime}^m$ (for more details on gluings see \cite[Chapter I.5]{brha}). Note that two spaces of constant curvature can only be glued through geodesic lines, geodesic segments or singletons. First we consider the gluing of two model spaces by a singleton. Although the reflection mapping is not uniquely defined in this type of gluing spaces, there exists a natural way of defining it such that it is nonexpansive.

\begin{pro} \label{prop_glu_point}
 Let $k,k^\prime \in \mathbb{R}$, $k \ne k^\prime$ and $m,n \in \mathbb{N}$. Consider the gluing space $(X,d)$ of $M_k^n$ and $M_{k^\prime}^m$ through a point $\theta$. Suppose $B$ is a nonempty closed and convex subset of $X$ and $x\in X$ such that $\emph{dist}(x,B)<\min\{D_k,D_{k^\prime}\}/2$. We define the reflection of a point with respect to $B$ in $X$ in the following way:
\begin{itemize}
\item [(1)] If $\theta \in B$, $R_Bx= R_{B\cap M_j^p} x$, where $M_j^p \in \{M_k^n, M_{k^\prime}^m\}$ is the model space that contains $x$ and $R_{B\cap M_j^p}$ is the reflection in the model space $M_j^p$ with respect to $B \cap M_j^p$.
\item [(2)] If $\theta \notin B$, observe that $B$ is strictly contained in just one of the two model spaces that define the glued space $X$. Suppose $B \subseteq M_k^n$. Then,
\begin{itemize}
\item [(2.1)] if $x \in M_k^n$, $R_B x$ is the point obtained by reflecting in $M_k^n$ with respect to $B$;
\item [(2.2)] if $x \in M_{k^\prime}^m$, the reflection is uniquely defined since $d(\theta, P_B\theta)>0$. Specifically, $R_B x$ is the point in the geodesic line in $M_k^n$ containing $[\theta, P_B \theta]$ such that $d(P_B \theta, R_B x)=d(x,P_B\theta)$ and $d(R_Bx,\theta)> d(R_Bx,P_B \theta)$.
\end{itemize}
\end{itemize}
Then $R_B$ is nonexpansive.
\end{pro}

\begin{proof}
(1) Suppose $\theta \in B$. We prove that $R_B$ defined as in (1) is nonexpansive. Let $x,y \in X$. The cases $x,y \in M_k^n$ or $x,y \in M_{k^\prime}^m$ are immediate since $R_B$ reduces to the reflection in model spaces. Suppose now $x\in M_k^n$ and $y\in M_{k^\prime}^m$. Then, $R_Bx= R_{B\cap M_k^n} x$ and $R_By= R_{B\cap M_{k^\prime}^m} y$. Consider the geodesic triangles $\triangle_1=\triangle(R_Bx,P_Bx,\theta)$ and $\triangle_2=\triangle(x,P_Bx,\theta)$. Notice that $\triangle_1,\triangle_2 \subseteq M_k^n$. For simplicity, we consider $k=-1$ (the proof for $k\neq -1$ follows similar patterns). Since $\angle_{P_Bx}(x,\theta) \geq \pi/2$, similarly as in the proof of Proposition \ref{model}, $d(\theta,x)\geq d(\theta, R_Bx)$. Likewise, we can see that $d(\theta, y)\geq d(\theta, R_B y)$.
Therefore,
$$d(R_Bx,R_By)=d(R_Bx,\theta)+d(\theta,R_By)\leq d(x,\theta)+d(y,\theta)=d(x,y),$$
and the conclusion follows.\\
(2) Next we suppose that $\theta \notin B$ and $B \subseteq M_k^n$. We prove that $R_B$ defined as in (2) is nonexpansive. Let $x,y \in X$. As we mentioned before, the case $x,y \in M_k^n$ is immediate. Let $x,y \in M_{k^\prime}^m$. Since $\text{dist}(\theta,B)>0$, we have that $P_Bx=P_By=P_B \theta$ and so
$$d(R_Bx,R_By)=|d(x,P_Bx)-d(y,P_By)|=|d(x,\theta)-d(y,\theta)|\leq d(x,y).$$
Finally, suppose $x\in M_k^n$, $y\in M_{k^\prime}^m$. As in the previous case, $P_B y=P_B \theta$. Consider $z=R_{P_B\theta} R_By \in M_k^n$, where $R_{P_B\theta}$ is the reflection mapping with respect to the point $P_B\theta$ inside $M_k^n$. By definition we have $P_B \theta \in [R_B y, z]$. Since in a model space geodesics do not bifurcate, the geodesic line containing $[R_B y, P_B \theta]$ coincides with the geodesic line containing $[R_B y, \theta]$. Let $c$ be this geodesic line. Then, by definition of the reflection, $z \in c$. Moreover, $d(R_By, P_B \theta)=d(z,P_B\theta) > d(P_B \theta, \theta)$ and $d(\theta, R_By) > d(R_By, P_B\theta)$. Thus, $\theta \in [P_B\theta, z]$. As a consequence,
$$d(P_B \theta, y)=d(P_B \theta, R_B y)=d(P_B\theta,z)=d(P_B \theta, \theta)+d(\theta, z).$$
On the other hand, $P_Bz=P_B\theta$ and therefore $R_Bz=R_By$. Finally, since $d(x,z)\leq d(x,\theta)+d(z,\theta)=d(x,y)-d(y,\theta)+d(z,\theta) = d(x,y)$ and the reflection in $M_k^n$ is nonexpansive, we get
$$d(R_Bx, R_By)=d(R_Bx,R_Bz)\leq d(x,z) \leq d(x,y),$$
and the result follows.
\end{proof}

Next we see that when gluing two model spaces through a geodesic segment or a geodesic line, we cannot define the reflection in the gluing in such a way that it is nonexpansive with respect to every nonempty closed and convex set. To illustrate this fact we consider the particular gluing of ${\mathbb{H}}^2$ and ${\mathbb R}^2$.

\begin{ex} Let $c \colon {\mathbb{R}} \rightarrow {\mathbb R}^2$ be the geodesic line in the plane defined by $c(t)=(t,0)$ and $c^\prime \colon {\mathbb{R}} \rightarrow {\mathbb H}^2$ the geodesic line in the hyperbolic plane defined by $c^\prime (t)= (0, \sinh t, \cosh t)$. We consider the gluing of these two model spaces by the correspondence given by the isometry $j \colon c({\mathbb{R}}) \to c^{\prime} ({\mathbb R})$ defined as $j(c(t))=c^\prime(t)$ for every $t \in \mathbb{R}$.

\noindent Let $x=(- \emph{arccosh} \sqrt{2},h)$ and $y=(\emph{arccosh} \sqrt{2},h)$, where $h>0$. Take $p=(\sqrt{z^2-1},0,z)$, with $z>1$ and $B=\{p\}$.

\noindent We see that there exist values of $h$ and $z$ such that $d(R_Bx,R_By)> d(x,y)=2 \emph{arccosh} \sqrt{2}$.
Notice that $$d(y,p)=\inf_{t \in {\mathbb R}} \{ d(y,c(t))+d(c^\prime(t),p)\}= \inf_{t \in {\mathbb R}} \bigg{\{}\sqrt{(\emph{arccosh} \sqrt{2}-t)^2 + h^2} + \emph{arccosh} (z \cosh t)\bigg{\}}.$$
Let $f(t)=d(y,c(t))+d(c^\prime(t),p)$. Then,
\[f^\prime(t)=\frac{z \sinh t}{\sqrt{z^2\cosh ^2 t-1}}-\frac{\emph{arccosh}\sqrt{2}-t}{\sqrt{(\emph{arccosh}\sqrt{2}-t)^2+h^2}}.\]
Since $\displaystyle \lim_{t \rightarrow \infty} f^\prime(t)=2$, $\displaystyle \lim_{t \rightarrow -\infty} f^\prime(t)=-2$ and $f^\prime$ is increasing in $t$, then there exists only one value $t_0 \in \mathbb{R}$ that gives the infimum. Moreover, since $f^\prime(0)<0$ and $\displaystyle f^\prime(\emph{arccosh}\sqrt{2})>0$, then $t_0 \in (0, \emph{arccosh} \sqrt{2})$.\\
By using the symmetry of the space, we see that if $t_0$ is the value where the previous infimum is attained, then $-t_0$ is the value that gives $d(x,p)$. Moreover, $d(x,p)=d(y,p)$.\\
Let $d=d(R_Bx,R_By)$ and $\gamma$ be the hyperbolic angle between the segments $[p,c(t_0)]$ and $[p, c(-t_0)]$. By the hyperbolic cosine law, we have
$$\cosh 2t_0=(\cosh d(p,c^\prime(t_0)))^2-(\sinh d(p,c^\prime(t_0)))^2 \cos \gamma$$ and
$$\cosh d = (\cosh d(x,p))^2-(\sinh d(x,p))^2 \cos \gamma.$$
If we consider for instance $h=1/100$ and $1<z=(\sqrt{2}+15)/16<\sqrt{2}$, then
$t_0\simeq 0.8392$ and $\cos \gamma \simeq -0.7991$. Consequently, $\cosh d\simeq 3.7363 > \cosh d(x,y)=3$.\\
The underlying idea in this construction is the fact that geodesics in ${\mathbb H}^2$ diverge faster than in ${\mathbb R}^2$.

\end{ex}
Since the previous construction can be easily adapted to the case of gluings through segments, we also obtain the same conclusion for this type of gluings.

\subsection{Convergence results}
We begin this section by defining the {\it averaged alternating reflection (AAR) method} in the setting of geodesic spaces. Let $(X,d)$ be a uniquely geodesic space. Given $A$ and $B$ two nonempty closed and convex subsets of $X$, suppose that the reflection mappings $R_A$ and $R_B$ are well-defined and singlevalued. Consider the mapping $T  \colon X \rightarrow X$ defined by $\displaystyle T=\frac{I+R_AR_B}{2}$. The AAR method generates the following sequence for a starting point $x_0 \in X$: $x_n = T^n x_0$ for every $n \ge 1$.

We prove next the convergence of the AAR method in the setting of model spaces with $k \leq 0$ which is an analogue of a weak convergence result in Hilbert spaces \cite[Fact 5.9]{bauscom1}. Note that we obtain strong convergence since the model spaces are proper metric spaces (see also \cite[Proposition 4.4]{bijan}; recall that a space is proper if every closed ball is compact).

\begin{theorem} \label{th_conv_neg}
Let $k \le 0$ and $n \in \mathbb{N}$. Suppose $A$ and $B$ are two nonempty closed and convex subsets of $M_k^n$ with $A \cap B \neq \emptyset$. Let $x_0 \in M_k^n$ and $(x_n)$ be the sequence starting at $x_0$ generated by the $AAR$ method. Then
\begin{itemize}
\item [(1)]$(x_n)$ converges to some fixed point $x$ of the mapping $T$ and $P_Bx \in A \cap B$.
\item [(2)] The ``shadow'' sequence $(P_Bx_n)$ is convergent and its limit belongs to $A \cap B$.
\end{itemize}
\end{theorem}

\begin{proof}
(1) By Proposition \ref{model}, $R_A R_B$ is nonexpansive which yields that $T$ is also nonexpansive because $M_k^n$ with $k\leq 0$ is Busemann convex. In addition, $\text{Fix}(T) \ne \emptyset$ since $A \cap B \ne \emptyset$. Thus, all orbits of $T$ are bounded. Using \cite[Proposition 2]{goki*} it follows that $T$ is asymptotically regular. Applying \cite[Proposition 6.3]{AriLeuLop12}, we get that $(x_n)$ $\Delta-$converges to a fixed point of $T$, which implies that $(x_n)$ converges to a fixed point of $T$. Let $x$ be the limit of $(x_n)$. Since $x=Tx$ it follows that $x=R_AR_Bx$. Moreover, because $\displaystyle P_B x=\frac{x + R_Bx}{2}$,
$$P_Bx=\frac{R_AR_Bx + R_Bx}{2}=P_A R_Bx \in A.$$
Hence, $P_B x \in A \cap B$.\\
(2) is immediate since the metric projection $P_B$ is continuous in CAT$(0)$ spaces.
\end{proof}

We remark that a similar result holds in the gluing spaces described in Proposition \ref{prop_glu_point} (when $k, k' \le 0$). In $\mathbb{H}^{\infty}$ one obtains an analogous $\Delta$-convergence result.

In model spaces $M_k^n$ with $k>0$ the mapping $T$ used in defining the AAR method may fail to be nonexpansive as the following example shows.

\begin{ex} Consider the points $a=(1,0,0)$, $b=(\sqrt{2}/2, \sqrt{2}/2,0)$ and $p=(0,0,1)$ in the spherical space ${\mathbb S}^2$. Let $c \in [a,p]$ and $c^\prime \in [b,p]$ such that $d(a,c)=d(b,c^\prime)=\pi/8$. Suppose also $A=[a,b]$ and $B=[c,c^\prime]$. To see that the mapping $T$ is not nonexpansive it is enough to note that $Tc=P_Ac$, $Tc^\prime=P_A c^\prime$, but $d(Tc,Tc^\prime)=d(P_Ac,P_A c^\prime)=d(a,b)=\pi/4 > d(c,c^\prime).$\\
Taking $B$ to be the whole positive octant, we obtain a similar example for the case $A \cap B \neq \emptyset$.
\end{ex}

However, for model spaces $M_k^n$ with $k>0$ we can prove the result below. Note that the existence of a convergent subsequence of $(x_n)$ already guarantees that $A \cap B \ne \emptyset$.

\begin{theorem} \label{th_conv_pos}
Let $k > 0$ and $n \in \mathbb{N}$. Suppose $A$ and $B$ are two nonempty closed and convex subsets of $M_k^n$. Let $C \subseteq M_k^n$ be nonempty convex with $\emph{diam}(C)<D_k/2$ such that $A, B \subseteq C$ and $R_AR_B(C) \subseteq C$. Let $x_0 \in C$ and $(x_n)$ be the sequence starting at $x_0$ generated by the $AAR$ method. Then,
\begin{itemize}
\item [(1)] Any convergent subsequence $(x_{n_k})$ of $(x_n)$ converges to a fixed point $x$ of the mapping $T$ and $P_Bx \in A \cap B$.
\item [(2)] The ``shadow'' sequence $(P_Bx_n)$ is bounded and each of its cluster points belongs to $A \cap B$.
\end{itemize}
\end{theorem}

\begin{proof}
(1) Since $\text{diam}(C)<D_k/2$, we know that the metric is convex on $C$. Reasoning as in the previous theorem, we get that $d(x_n,x_{n+1}) \to 0$ as $n \to \infty$. Let $(x_{n_k})$ be a convergent subsequence of $(x_n)$ and $x$ be its limit. Since $R_AR_B$ is continuous we have $R_AR_Bx_{n_k} \rightarrow R_AR_Bx$. Besides, $(x_{n_k+1})$ also tends to $x$, which implies that $x=Tx$. The fact that $P_B x \in A \cap B$ follows as in the proof of the previous theorem.\\
(2) Let $(P_Bx_{n_k})$ be a convergent subsequence of $(P_Bx_n)$ and $y$ be its limit. We may suppose that $(x_{n_k})$ is convergent (otherwise consider a convergent subsequence of it) and denote $\displaystyle x=\lim_{k\to \infty} x_{n_k}$. From the continuity of the metric projection $P_B$, we get that $y=P_Bx$ and we are done.
\end{proof}

A more general way of defining the sequence $(x_n)$ would be to use the Krasnoselski-Mann \cite{Mann53, Kra55} iteration starting at $x_0 \in X$:
\begin{equation} \label{eq_KM}
x_{n+1} = (1-\lambda_n)x_n + \lambda_n R_AR_B x_n,
\end{equation}
where $(\lambda_n) \subseteq [0,1]$.

Results similar to Theorems \ref{th_conv_neg}, \ref{th_conv_pos} hold for the sequence $(x_n)$ generated by (\ref{eq_KM}) when assuming for instance that $(\lambda_n)$ is divergent in sum and bounded away by $1$. In this case, consider $F=\text{Fix}(R_AR_B)\supseteq A\cap B $ and note that $(d(x_n,p))$ is decreasing for each $p \in F$. Apply \cite[Proposition 2]{goki*} as above to get that $d(x_n, R_AR_B x_n) \to 0$ (if $k<0$, use asymptotic center techniques similar to those considered in \cite[Proposition 6.3]{AriLeuLop12} to obtain the convergence of $(x_n)$). Notice that now one obtains convergence to a fixed point of the mapping $R_AR_B$ (see also \cite{DhoPan08}).

We finish this section by giving a rate of asymptotic regularity for the sequence $(x_n)$ generated by (\ref{eq_KM}). Using proof mining methods, Kohlenbach \cite{Koh01} and later Kohlenbach and Leu\c stean \cite{KohLeu03} computed exponential (in $1/\varepsilon)$ rates of asymptotic regularity for the Krasnoselski-Mann iteration in normed and hyperbolic spaces, respectively.  The next result which follows from \cite[Corollary 3.18]{KohLeu03} gives an explicit bound on the rate of asymptotic regularity for the sequence $(x_n)$.

\begin{theorem}
Let $k \in \mathbb{R}$ and $n \in \mathbb{N}$. Suppose $A$ and $B$ are two nonempty closed and convex subsets of $M_k^n$ with $A \cap B \ne \emptyset$. Let $C \subseteq M_k^n$ be nonempty convex and of diameter bounded above by $b \in (0,D_k/2)$ such that $A, B \subseteq C$ and $R_AR_B(C) \subseteq C$. Assume $K \in \mathbb{N}$, $K \ge 2$ and $\displaystyle (\lambda_n) \subseteq \left[\frac{1}{K}, 1-\frac{1}{K}\right]$. Then,
\[\forall x_0 \in X, \forall \varepsilon > 0, \forall n \ge \Phi(\varepsilon, b, K), \quad d(x_n, x_{n+1}) \le \varepsilon,\]
where
\[\Phi(\varepsilon, b,K) := KM\left\lceil2b e^{K(M+1)} \right\rceil,\]
with
\[M \ge \frac{(K-1)(1+2b)}{K\varepsilon}.\]
\end{theorem}

Thus, the above rate only depends on $\varepsilon$, on an upper bound $b$ on the diameter of the set $C$ and on $(\lambda_n)$ via $K$. Rates of asymptotic regularity for the Krasnoselski-Mann iteration were further studied and improved by Leu\c stean \cite{Leu07} in the setting of uniformly convex hyperbolic spaces that admit a modulus of uniform convexity which decreases with respect to the radius. As a consequence, one gets a quadratic in $1/\varepsilon$ rate of asymptotic regularity in CAT$(0)$ spaces for constant $\lambda_n = \lambda \in (0,1)$ \cite[Corollary 19]{Leu07}. In particular, these results can be applied in the setting of model spaces $M_k^n$ with $k\le0$.

\section{Appendix: A related property of model spaces}
We have seen above that the reflection mapping with respect to nonempty closed and convex subsets is nonexpansive in the context of model spaces. The property we prove below implies in particular that the reflection mapping is nonexpansive in all model spaces $M_k^n$ with $k < 0$.

\begin{pro}
Let $k \in \mathbb{R}$, $k \ne 0$ and $n \in \mathbb{N}$. Consider $C \subseteq M_k^n$ with $\emph{diam}(C) < D_k$ and the points $x,y,x',y',a,b \in C$ for which
\[\displaystyle a = \frac{1}{2}x + \frac{1}{2}x' \quad \text{and} \quad \displaystyle b = \frac{1}{2}y + \frac{1}{2}y'.\]
Denote by $\gamma$ and $\gamma'$ the angles between the geodesic segments $[y,b]$ and $[b,a]$, $[x,a]$ and $[a,b]$. Suppose that
\begin{equation} \label{eq_appendix}
d(a,b) \le d\left((1-\lambda)x + \lambda a, (1-\lambda)y + \lambda b\right), \quad \text{for each }\lambda \in [0,1].
\end{equation}
If one of the following holds
\begin{itemize}
\item[\emph{(i)}] $d(x,a) = 0$,
\item[\emph{(ii)}] $d(x,a) = d(y,b)$,
\item[\emph{(iii)}] $d(x,a) < d(y,b)$ and ($\gamma \le \pi/2$ or $\gamma' \ge \pi/2$ for $k < 0$), ($\gamma \ge \pi/2$ or $\gamma' \le \pi/2$ for $k > 0$),
\end{itemize}
then $d(x',y') \le d(x,y)$.
\end{pro}
\begin{proof}
We prove the result for ${\mathbb{H}}^n$, the proof for the other model spaces can be obtained by using the corresponding law of cosines.\\
For $\lambda \in (0,1)$, consider
\[x_\lambda = (1 - \lambda)x + \lambda a \quad  \text{and} \quad y_\lambda = (1-\lambda)y + \lambda b.\]
Denote $d(x,a) = A$ and $d(y,b) = B$.

(i): By the hyperbolic cosine law,
\[\cosh d(a,b) \le \cosh d(a, y_\lambda) = \cosh d(a,b)\cosh d(y_\lambda,b) - \sinh d(a,b)\sinh d(y_\lambda,b) \cos \gamma.\]
Thus,
\[\sinh d(a,b)\cos \gamma \le \frac{\cosh \left((1-\lambda)B\right) - 1}{\sinh \left((1-\lambda)B\right)}\cosh d(a,b).\]
Letting $\lambda \to 1$ we obtain that $\cos \gamma \le 0$. Using again the hyperbolic cosine law it follows that $\cosh d(a, y') \le \cosh d(a,y)$, hence, $d(a,y') \le d(a,y)$.

(ii) or (iii): Apply again the hyperbolic cosine law in the triangles $\Delta(a,y,b)$ and $\Delta(a,y',b)$ to obtain that
\begin{equation}\label{prop_S_eq1}
\cosh d(a,y) + \cosh d(a,y') = 2\cosh d(a,b) \cosh B.
\end{equation}
By the hyperbolic cosine law applied in the triangles $\Delta(x_\lambda,y, y_\lambda)$ and $\Delta(x_\lambda, y', y_\lambda)$,
\begin{equation} \label{prop_S_eq2}
\sinh \left((2-\lambda)B\right)\cosh d(x_\lambda, y) + \sinh \left(\lambda B\right) \cosh d(x_\lambda, y') = \sinh \left(2B\right)\cosh d(x_\lambda, y_\lambda).
\end{equation}
Using again the hyperbolic cosine law in the triangles $\Delta(a,x,y)$ and $\Delta(x_\lambda, x, y)$,
\[\frac{\cosh d(a,y) - \cosh d(x,y) \cosh A}{\sinh A} = \frac{\cosh d(x_\lambda, y) - \cosh d(x,y)\cosh \left(\lambda A\right)}{\sinh \left(\lambda A\right)},\]
from where
\[\sinh \left(\lambda A\right)\cosh d(a,y) + \sinh \left((1-\lambda)A\right)\cosh d(x,y) = \sinh A \cosh d(x_\lambda,y).\]
Applying the same reasoning in the triangles $\Delta(a, x', y')$ and $\Delta(x_\lambda, x', y')$ we obtain
\[\sinh \left((2-\lambda) A\right)\cosh d(a,y') = \sinh A\cosh d(x_\lambda,y') + \sinh \left((1-\lambda)A\right) \cosh d(x',y').\]
This implies
\begin{align*}
& \sinh \left(\lambda A\right) \sinh \left((2-\lambda)B\right) \left(\cosh d(a,y) +  \cosh d(a,y') \right)\\
& \ \ \ \  + \cosh d(a,y') \left(\sinh \left((2-\lambda) A\right) \sinh \left(\lambda B \right) - \sinh \left(\lambda A\right) \sinh \left((2-\lambda)B\right)\right) \\
& \ \ \ \ + \sinh \left((1-\lambda)A\right) \sinh \left((2-\lambda)B\right) \cosh d(x,y)\\
& \ \ = \sinh A \left( \sinh \left((2-\lambda)B\right)\cosh d(x_\lambda, y) + \sinh \left(\lambda B\right) \cosh d(x_\lambda, y') \right)\\
& \ \ \ \  + \sinh \left((1-\lambda)A\right) \sinh \left(\lambda B \right)\cosh d(x',y').
\end{align*}
Using (\ref{prop_S_eq1}) and (\ref{prop_S_eq2}) and the fact that $\cosh d(x_\lambda, y_\lambda) \ge \cosh d(a,b)$ we obtain
\begin{align*}
&2\cosh B \cosh d(a,b) \left(\sinh \left(\lambda A\right) \sinh \left((2-\lambda)B\right)  - \sinh A\sinh B\right)\\\
& \ \ \ \  +  \cosh d(a,y') \left(\sinh \left((2-\lambda) A\right) \sinh \left(\lambda B \right) - \sinh \left(\lambda A\right) \sinh \left((2-\lambda)B\right)\right) \\
& \ \ \ \ + \sinh \left((1-\lambda)A\right) \sinh \left((2-\lambda)B\right) \cosh d(x,y)\\
& \ \ \ge \sinh \left((1-\lambda)A\right) \sinh \left(\lambda B \right)\cosh d(x',y').
\end{align*}
Dividing by $\sinh \left((1-\lambda)A\right)$ and letting $\lambda \to 1$,
\begin{align*}
&2\left(\cosh B \cosh d(a,b) - \cos d(a,y')\right)\left(\frac{B}{A}\sinh A \cosh B  - \cosh A \sinh B\right)\\
& \ \ \ge \sinh B \left(\cosh d(x',y') - \cosh d(x,y)\right).
\end{align*}
Therefore,
\begin{align*}
2\sinh d(a,b)  \cos (\pi - \gamma)\left(\frac{B}{A}\sinh A \cosh B  - \cosh A \sinh B\right) \ge \cosh d(x',y') - \cosh d(x,y).
\end{align*}
In a similar way we obtain that
\begin{align*}
2\sinh d(a,b)  \cos (\pi - \gamma')\left(\frac{A}{B}\sinh B \cosh A  - \cosh B \sinh A\right) \ge \cosh d(x',y') - \cosh d(x,y).
\end{align*}
Since $\displaystyle x \mapsto \frac{\tanh x}{x}$ is non-increasing, the conclusion follows.\\
\end{proof}

Note that when $k=0$ the above property is immediate. In fact, in any Hilbert space, it is easy to see that condition (\ref{eq_appendix}) already implies that $d(x',y') \le d(x,y)$ (see also the so-called property (S) with $b=1$ studied in \cite{BruRei77} which is equivalent to this condition in the setting of normed spaces).

\section*{Acknowledgements}
Aurora Fern\' andez-Le\' on was partially supported by the Plan Andaluz de Investigaci\' on de la Junta de Andaluc\'ia FQM-127 and Grant P08-FQM-03543, and by MEC Grant MTM2009-10696-C02-01. Part of this work was carried out while she was visiting the Babe\c s-Bolyai University in Cluj-Napoca. She acknowledges the kind hospitality of the Department of Mathematics.

Adriana Nicolae was supported by a grant of the Romanian National Authority for Scientific Research, CNCS-UEFISCDI, project number PN-II-ID-PCE-2011-3-0383.

\end{document}